\documentclass[11pt,thmsa]{article}
\usepackage{amsmath, latexsym, amsfonts, amssymb, amsthm, amscd}
\usepackage{color}
\newtheorem{theorem}{Theorem}
\newtheorem{corollary}{Corollary}
\newtheorem{proposition}{Proposition}
\newtheorem{lemma}{Lemma}

\newtheorem{defi}{Definition}
\newcommand{\p}{\Bbb{P}}

\newcommand{\e}{\Bbb{E}}

\newcommand{\R}{\mathbb{R}}

\newcommand{\ud}{\mathrm{d}}

\begin{document}

\title{Refracted Continuous-State Branching Processes:
Self-regulating populations}

\maketitle

\begin{center}

{\sc A. Murillo-Salas\footnote{ {Departamento de Matem\'aticas, Universidad de Guanajuato.} E-mail: amurillos@ugto.mx},  J.L. P\'erez\footnote{ {CIMAT.} E-mail: jluis.garmendia@cimat.mx} and A. Siri-J\'egousse\footnote{ {IIMAS, Universidad Nacional Aut\'onoma de M\'exico.} E-mail: arno@sigma.iimas.unam.mx} }
 \end{center}
\vspace{0.2in}

\begin{abstract}{ 
 We construct a modified continuous-state branching process whose Malthusian parameter is replaced by another one when passing below a certain level. 
The construction is obtained via a Lamperti-like transform applied to a refracted L\'evy process. 
Infinitesimal generator, probability of vanishing at infinity, of explosion and some path properties are also provided.}
\noindent
\bigskip


\bigskip

\end{abstract}

\vspace{0.5cm}

\maketitle

\section{Introduction}


Naturally reproducing populations, when reaching a low level, might try to struggle not to extinct by increasing their reproduction rate.
Imagine a population model with (global) self-regulation such that whenever the population size goes below a certain value, every individual will equitably increase its reproduction success.

Size-dependent continuous-state branching processes have been the object of many studies in the literature (see \cite{JK} and references therein).
This paper deals with a very precise case where it is possible to provide explicit  results in terms of the scale function of an associated L\'evy process.
In particular, it may be of reader's interest to notice that this paper treats with continuous-state branching processes with discontinuous size-dependent Malthusian parameter.
 As a consequence of this fact,  it is not possible to characterise such processes as the strong solution of a stochastic differential equation, as it is described for example in Li and Pu \cite{LP}.

As mentioned before, our results strongly rely on the theory  of scale functions for spectrally positive L\'evy processes, which we shall introduce now. 
Let $X=\{X_t, t\geq 0\}$ be a L\'evy process defined on a  probability space $(\Omega, \mathcal{F}, \p)$.  For each $x\in \R$ denote by $\p_x$ the law of $X$ when it is started at $x$ and write for convenience  
$\p$ in place of $\p_0$. Accordingly, we shall write $\e_x$ and $\e$ for the associated expectation operators. In this paper we shall assume throughout that $X$ is \textit{spectrally negative} meaning here that it has 
no positive jumps and that it is not { the negative} of a subordinator. It is well known that the latter allows us to talk about the Laplace exponent, $\psi:[0,\infty) \to \R$, i.e.\[
\e\Big[{\rm e}^{\theta X_t}\Big]:={\rm e}^{t\psi(\theta)}, \qquad t, \theta\ge 0,
\]
and the Laplace exponent $\psi$ is given by the L\'evy-Khintchine formula
\begin{equation*}
\psi(\theta)={-\beta}+\gamma\theta+\frac{\sigma^2}{2}\theta^2+\int_{(-\infty,0)}\big({\rm e}^{\theta x}-1-{\theta x{1}_{\{x>-1\}}}\big)\Pi(\ud x),
\end{equation*}
where $\gamma\in \R$, ${\beta},\sigma^2\ge 0$ and $\Pi$ is a measure on $(-\infty,0)$ called the L\'evy measure of $X$ and satisfies
\[
\int_{(-\infty,0)}(1\land x^2)\Pi(\ud x)<\infty.
\]
The reader is referred to Bertoin \cite{B} and Kyprianou \cite{K} for a complete introduction to the theory of L\'evy processes.

It is well-known that $X$ has paths of bounded variation if and only if $\sigma^2=0$ and $\int_{(-1, 0)} x\Pi(\mathrm{d}x)$ is finite. In this case $X$ can be written as
\begin{equation}
X_t=x+ct-S_t, \,\,\qquad t\geq 0,
\label{BVSNLP}
\end{equation}
where {$c=\gamma-\int_{(-1,0)} x\Pi(\mathrm{d}x)$} and {$S=(S_t,t\geq0)$} is a driftless subordinator {killed at rate $\beta$}. Note that  necessarily $c>0$, since we have ruled out the case that $X$ has monotone paths. In this case its Laplace exponent is given by
\begin{equation*}
\psi(\theta)= \log \mathbb{E} \left[ \mathrm{e}^{\theta X_1} \right] = {-\beta}+c \theta-\int_{(-\infty,0)}\big(1- {\rm e}^{\theta x}\big)\Pi(\ud x),\,\,\theta\geq0.
\end{equation*}

Let $\delta$ be a positive constant such that 
\begin{equation}
\mathrm{({\bf H})}\qquad\delta<c=\gamma-\int_{(-1,0)} x\Pi(\mathrm{d}x)\qquad\text{when $X$ has paths of bounded variation.}\notag
\end{equation}
Now consider the process $U=\{U_t,t\geq0\}$ which is a solution to the following stochastic differential equation:
{$U_0=-x$ and}
\begin{equation}\label{SDE}
d U_t=d X_t-\delta\mathbf{1}_{\{U_t>-b\}}d t, \qquad t> 0,
\end{equation}
i.e., a linear drift at rate $\delta>0$ is subtracted from the increments of
$X$ whenever it exceeds a pre-specified level $-b$.
According to Kyprianou and Loeffen \cite{KL} Theorem 1, condition ({\bf H})  ensures that there exists a unique strong solution to \eqref{SDE} and the paths of $U$ are not monotone.

Now let us consider $b>0$, and let us assume that $\overline{X}$ is a 
 spectrally positive L\'evy process started from $x>0$ 
with Laplace exponent
 \begin{equation}\label{lk}
\overline\psi(\theta)={-\beta}-\gamma\theta+\frac{\sigma^2}{2}\theta^2+\int_{(0,\infty)}\big({\rm e}^{\theta x}-1-\theta x{1_{\{x<1\}}}\big)\overline\Pi(\ud x),
\end{equation}
where $\gamma\in \R$, ${\beta},\sigma^2\ge 0$ and $\overline\Pi(\mathrm{d}x)=\Pi(\mathrm{d}(-x))$ is a measure on $(0,\infty)$ which satisfies
\[
\int_{(0,\infty)}({1}\land x^2)\overline\Pi(\ud x)<\infty.
\]
Let us assume that the following condition holds
\begin{equation}
\mathrm{(\overline{{\bf H}})}\qquad\delta<c=\gamma+\int_{(0,1)} x\overline{\Pi}(\mathrm{d}x)\qquad\text{when $\overline{X}$ has paths of bounded variation.}\notag
\end{equation}
In this case we can define the process $U$ as a solution to (\ref{SDE}) driven by $-\overline{X}$. That is {$U_0=-x$ and}
$$
d U_t=-d {\overline X}_t-\delta\mathbf{1}_{\{U_t>-b\}}d t, \qquad t> 0.
$$
This follows from the fact that $-\overline{X}$ is a spectrally negative L\'evy process. Let us define $\overline{U}\equiv -U$,  it follows that
\begin{align}
\overline{U}_t&=\overline{X}_t+\delta\int_0^t\mathbf{1}_{\{U_s>-b\}}ds=\overline{X}_t+\delta\int_0^t\mathbf{1}_{\{\overline{U}_s<b\}}ds\notag.
\end{align}
This implies that, for any spectrally positive L\'evy process $\overline{X}$ started from $x$, there exists a  unique strong solution to the following stochastic differential equation: $\overline{U}_0=x$
\[d\overline{U}_t=d\overline{X}_t+\delta\mathbf{1}_{\{\overline{U}_t<b\}}dt,\qquad t>0,\]
equivalently 
\begin{equation}\label{SDE2}
\overline{U}_t=\overline{X}_t+\delta\int_0^t\mathbf{1}_{\{\overline{U}_s<b\}}ds, \qquad t\geq 0,
\end{equation}
thus any time the process $\overline{U}$ is below level $b$ a positive drift $\delta$ is added.

The paper is organised as follows. In Section \ref{RefractedSection} we formally define refracted continuous-state branching processes (RCSBP). In Section \ref{extexp} we   study useful classical  functionals such as the law of the total progeny, the probability of extinction and explosion. Finally, in Section \ref{illmj}
we investigate a law of iterated logarithm and the maximal jump of RCSBP.


\section{Refracted continuous-state branching processes}\label{RefractedSection}

\subsection{The Lamperti-like representation}\label{rcsbpLamperti}

Recall that continuous-state branching processes are the continuous time and space version of the classical  Galton-Watson processes, which are Feller processes with state space $[0,\infty]$ such that the logarithm of the Laplace transform of its transition semigroup is given by an affine transformation of the initial state. To be more precise, let $Z=\{Z_t:t\geq 0\}$ be a continuous-state branching process, and then
{\begin{equation}
\frac{d}{dt}\log\left(\mathbb{E}_x[e^{-\lambda Z_t}]\right)\bigg|_{t=0}=x \overline\psi(\lambda),\qquad\text{$\lambda>0$}.\notag
\end{equation}
In fact, by the classical Lamperti transformation \cite{la1,cpu}  ${\overline\psi}$  must be  the Laplace exponent of a spectrally positive L\'evy process.  This is the content of the next theorem, which we state   without proof.  
 Among other things, it shows that for every such Laplace exponent $\overline\psi$ there exists an associated continuous-state branching process.
\begin{theorem}
Let $x\geq 0$ and $\overline{X}$ be a spectrally positive L\'evy process with Laplace exponent ${\overline\psi}$ defined in \eqref{lk}. The unique stochastic process $Z$ which solves
\begin{equation}\label{LT1}
Z_t=\overline{X}_{\int_0^tZ_sds},\qquad\text{$t\geq0$,}
\end{equation}
is a continuous-state branching process with branching mechanism ${\overline\psi}$ that starts at $x$.
\end{theorem}

Now, we apply the idea of the Lamperti transformation  to the process $\overline{U}$ provided  by (\ref{SDE2}) and construct a new process $V=\{V_t:t\geq0\}$ which will be  called    refracted continuous-state branching process. We formally introduce this construction in the  following definition.  
 
\begin{defi}
Let $\overline{U}$ be the process given by the unique solution to the stochastic differential equation given by (\ref{SDE2}) with $b>0$, then consider  the process $V=\{V_t:t\geq0\}$ given by:
\begin{equation}\label{rcsbp}
V_t=\overline{U}_{\int_0^tV_sds},\qquad\text{$t\geq0$}.
\end{equation}
We call the process $V$ \textbf{refracted continuous-state branching process} (RCSBP).
\end{defi}
The construction of the  stochastic process that satisfies \eqref{rcsbp} can be done by the method of time changes: let $\tau$ be the first hitting time of zero by the process $\overline{U}$, and let
\begin{equation}
i_t{=}\int_0^{\tau\wedge t}\frac{1}{\overline{U}_s}ds,\notag
\end{equation}
then consider its right-continuous inverse $c$.  Note that, 
\begin{equation}\label{totalprog}
c_t=\int_0^tV_s\,ds\equiv J_t,
\end{equation}
where $J_t$ is called the total progeny up to time $t$ which will be studied in Section \ref{tpext}. Therefore,  
 if we define
\begin{equation}
V=\overline{U}\circ c,\notag
\end{equation}
then $V$ satisfies \eqref{rcsbp}, and it is the unique solution for which zero is absorbing. A proof of this result for the case $\delta=0$ is given by Theorem 1 in \cite{cpu}, in this  case the transformation which takes $\overline U$ to $V$ is called the Lamperti transformation, introduced by Lamperti \cite{la1}. The proof of the result for $\delta>0$ follows the same lines as the above mentioned theorem, for sake of brevity we omit the details.

\subsection{The infinitesimal generator}\label{ig}
In this section we compute the form of the infinitesimal generator for the RCSBP  given by (\ref{rcsbp}). To do this, let $f\in\mathcal{C}^2_0(\mathbb{R}_+)$ (the space of non-negative functions vanishing at infinity with continuous second derivative)  and consider the operator
\begin{equation}
\Gamma f(x):={-\beta f(x)-}\gamma f'(x)+\frac{\sigma^2}{2}f''(x)+\int_0^{\infty}(f(x+z)-f(x)-zf'(x)){\overline\Pi}(dz).\notag
\end{equation}
We have the following result:
\begin{lemma}
{Let $\tilde X, \tilde U, \tilde J$ and $\tilde V$ be the analogues of $\overline X, \overline U, J$ and $V$ but without the killing component, i.e.,  $\tilde X$ has Laplace exponent
$$\tilde\psi(\theta)=-\gamma\theta+\frac{\sigma^2}{2}\theta^2+\int_{(0,\infty)}\big({\rm e}^{\theta x}-1-\theta x{1_{\{x<1\}}}\big)\overline\Pi(\ud x),
$$
where $\tilde U$ is obtained applying \eqref{SDE2} to $\tilde X$, $\tilde V=\tilde U\circ \tilde J$ and $\tilde J_t=\int_0^t\tilde V_sds$.}
For each $f\in\mathcal{C}^2_0(\mathbb{R}_+)$, the process $ N^f=\{ N^f_t,t\geq0\}$ given by 
\begin{equation}\label{mg}
{N}^f_t={e^{-\beta \tilde J_t}f(\tilde V_t)-\int_0^t e^{-\beta \tilde J_s}\tilde V_s\Gamma f(\tilde V_s)ds-\delta\int_0^te^{-\beta \tilde J_s}\tilde V_sf'(\tilde V_s)1_{\{\tilde V_s\leq b\}}ds,\quad t\geq0,}
\end{equation}
is a {{$\mathcal{F}_{\tilde J_t}$}-martingale}.
\end{lemma}
\begin{proof}
Let $f\in\mathcal{C}^2_0(\mathbb{R}_+)$. Using the fact that the process $\tilde {U}$ defined in (\ref{SDE2}) is a semimartingale we can apply It\^o's formula to the process ${e^{-\beta s}f(\tilde {U})}$ (cf. Theorem II.32 in \cite{Pr}), to obtain
\begin{eqnarray*}
{e^{-\beta t}}f(\tilde {U}_t)&=&f(\tilde {U}_0){-\beta\int_0^t{e^{-\beta s}}f(\tilde {U}_{s-})ds}+\int_0^t{e^{-\beta s}}f'(\tilde {U}_{s-})d\tilde {U}_s{+\frac{1}{2}\int_0^t {e^{-\beta s}}f^{\prime\prime}(\tilde {U}_{s^{-}})d[\tilde {U},\tilde {U}]_s^c}
\\&&+\sum_{0<s\leq t}{e^{-\beta s}}\left[f(\tilde {U}_{s-}+\Delta\tilde {X}_s)-f(\tilde {U}_{s-})-f'(\tilde {U}_{s-})\Delta \tilde {X}_s\right],
\end{eqnarray*}
where  $\Delta \tilde {X}_s:=\tilde {X}_s-\tilde {X}_{s-}$ and $[\tilde {U},\tilde {U}]_.^c$ denotes de continuous part of the quadratic variation process $[\tilde {U},\tilde {U}]_.$. Rewriting the above equation leads to
\begin{align}
{e^{-\beta t}}f(\tilde {U}_t)&=f(\tilde {U}_0)+\int_0^t{e^{-\beta s}}\Gamma f(\tilde {U}_{s-})ds+\delta\int_0^t{e^{-\beta s}}f'(\tilde {U}_{s-})1_{\{\tilde {U}_{s-}\leq b\}}ds\notag\\
&{+\left\{\int_0^t{e^{-\beta s}}f'(\tilde {U}_{s-})d[\tilde {X}_s+(\gamma-\int_1^\infty z\overline\Pi(dz))s]-\sum_{0<s\leq t}{e^{-\beta s}}f'(\tilde {U}_{s-})\Delta\tilde {X}_s1_{|\Delta\tilde {X}_s|>1}\right\}}\notag\\
&+\bigg\{\sum_{0<s\leq t}{e^{-\beta s}}(f(\tilde {U}_{s-}+\Delta \tilde {X}_s)-f(\tilde {U}_{s-})-f'(\tilde {U}_{s-})\Delta \tilde {X}_s1_{|\Delta \tilde {X}_s|\leq1})\notag\\
&-\int_0^t\int_0^{\infty}{e^{-\beta s}}(f(\tilde {U}_{s-}{+z})-f(\tilde {U}_{s-})-f'(\tilde {U}_{s-})z1_{\{z\leq 1\}}){\overline\Pi}(dz)ds\bigg\}.\notag
\end{align}
By the L\'evy-It\^o decomposition  the expression between the first pair of curly brackets is a local martingale and by the compensation formula (cf. \cite{K} Corollary 4.6) the expression between the second pair of curly brackets is also a local martingale.
This implies that
\begin{equation}
M^f_t={e^{-\beta t}}f(\tilde {U}_t)-\int_0^t{e^{-\beta s}}\Gamma f(\tilde {U}_s)ds-\delta\int_0^t{e^{-\beta s}}f'(\tilde {U}_s)1_{\{\tilde {U}_s\leq b\}}ds,\notag
\end{equation}
is a $\mathcal{F}^{\tilde  X}_t$-local martingale. Recalling   that $f\in\mathcal{C}^2_0(\mathbb{R}+)$, we get that $M^f$ has sample paths uniformly bounded on compact sets which in turn implies that $M^f$ is a martingale.
So, using that $\tilde J_t$ is a $\mathcal{F}^{\tilde  X}$-stopping time, then by optional stopping we have that
\begin{align}
N^f_t:=M^f_{\tilde J_t}&=f(\tilde {U}_{\tilde J_t}){e^{-\beta \tilde J_t}}-\int_0^{\tilde J_t}{e^{-\beta s}}\Gamma f(\tilde {U}_s)ds-\delta\int_0^{\tilde J_t}{e^{-\beta s}}f'(\tilde U_s)1_{\{\tilde {U}_s\leq b\}}ds\notag\\
&=f(\tilde V_t){e^{-\beta \tilde J_t}}-\int_0^{t}{e^{-\beta \tilde J_s}}\tilde V_s\Gamma f(\tilde V_s)ds-\delta\int_0^t{e^{-\beta \tilde J_s}}\tilde V_sf'(\tilde V_s)1_{\{\tilde V_s\leq b\}}ds,\notag
\end{align}
is a martingale with respect to the filtration $(\mathcal{F}_{J_t})_{t\geq 0}$.
\end{proof}
Now let us {get back to the refracted continuous-state branching process V and}  take $f\in \mathcal{H}=\{f\in\mathcal{C}^2_0(\mathbb{R}_+): f'(b)=0\}$. 
Then, {taking expectation and splitting into the events:   $V_t$ has been killed before time $t$ or not, we have by (\ref{mg})  that}
\begin{equation}\label{g2}
\mathbb{E}_x[f(V_t)]-f(V_0)=\int_0^t\mathbb{E}[V_s\Gamma f(V_s)+\delta V_sf'(V_s)1_{\{V_s\leq b\}}]ds.
\end{equation}
Then, by differentiating (\ref{g2}) with respect to $t$ and letting  $t\downarrow0$ we obtain the following result:
\begin{proposition} \label{Infinitesimalgenerator2}
 Let $f \in \mathcal{H}$. Then $f$ lies in the domain of the generator $\mathcal{A}$ of the process $V$ and
\begin{equation*}
\mathcal{A}[f](x)
=x\left[\Gamma f(x)+\delta\mathbf{1}_{\{x\leq b\}} f^\prime(x)\right].
\end{equation*}
\end{proposition}

\section{Extinction and explosion for RCSBP}\label{extexp}

\subsection{{Extinction in the diffusion case}}\label{extdiff}
In this section we will compute in an explicit way the probability of extinction for the case where  the refracted continuous-state branching process $V$ is
derived from a Brownian motion, i.e., in the case when the process $\overline{X}$ in equation (\ref{SDE2}) is a Brownian motion with drift. 
Let us define
\[T_0^-=\inf\{t>0:V_t=0\}\]
and denote by $p(x)$ the probability that the refracted continuous-state branching process $V$, started form $x>0$, becomes extinct, i.e., $p(x)=\mathbb{P}_{x}(T_0^-<\infty)$.

We use Proposition \ref{Infinitesimalgenerator2} to find an ordinary differential equation for the extinction probability. The result is as follows:
\begin{proposition}
Assume that $\overline{X}=\gamma t+\frac{\sigma^2}{2}B_t$ in equation (\ref{SDE2}), where $B$ is a standard  Brownian motion, then the probability of extinction of $V$ under $\mathbb P_x$  is given by:
\begin{align}
p(x)&=1+K\bigg(\frac{\sigma^2}{2(\gamma+\delta)}\left(1-e^{-\frac{2(\gamma+\delta)}{\sigma^2}(x\wedge b)}\right)+\frac{\sigma^2}{2\gamma}e^{-\frac{2\delta}{\sigma^2} b}\left(e^{-\frac{2\gamma}{\sigma^2}b}-e^{-\frac{2\gamma}{\sigma^2}(x\vee b)}\right)\bigg),\notag
\end{align}
where $K=-{\left(\frac{\sigma^2}{2(\gamma+\delta)}\left(1-e^{-\frac{2(\gamma+\delta)}{\sigma^2}b}\right)+\frac{\sigma^2}{2\gamma}e^{-\frac{2(\gamma+\delta)}{\sigma^2}b}\right)^{-1}}$.
\end{proposition}
\begin{proof}
Let us denote by $P_t$ the semigroup of the  RCSBP $V$, then {by an application of the Markov property}
\begin{equation}
P_tp(x)=\mathbb{E}_x[\mathbb{P}_{V_t}(T_0^-<\infty)]=\mathbb{E}_x[\mathbb{P}_{x}(T_0^-<\infty|\mathcal{F}^{\overline X}_{J_t})]=\mathbb{P}_{x}(T_0^-<\infty)=p(x).\notag
\end{equation}
Then, $x\mapsto p(x)$ is an invariant function for the semigroup $\{P_t,t\geq0\}$ which yields  that $p$ belongs to the domain of the generator $\mathcal{A}$ of
the process $V$, and that $\mathcal{A}[p](x)\equiv 0$.

In order to find the explicit form of $p$ we have to solve the following ordinary differential equation:
\begin{equation}\label{ode}
(\gamma+\delta 1_{\{x\leq b\}})xp'(x)+\frac{\sigma^2}{2}xp''(x)=0,
\end{equation}
by integrating (\ref{ode}) we see that the following holds:
\begin{equation}
p'(x)=p'(0)e^{-\frac{2(\gamma x+\delta(x\wedge b))}{\sigma^2}}.\notag
\end{equation}
Finally, by doing a second integration we obtain:
\[
p(x)=p(0)+p'(0)\bigg(\frac{\sigma^2}{2(\gamma+\delta)}\left(1-e^{-\frac{2(\gamma+\delta)}{\sigma^2}(x\wedge b)}\right)+\frac{\sigma^2}{2\gamma}e^{-\frac{2\delta}{\sigma^2} b}\left(e^{-\frac{2\gamma}{\sigma^2}b}-e^{-\frac{2\gamma}{\sigma^2}(x\vee b)}\right)\bigg).
\]
To conclude the proof we observe that $p(0)=1$, and $\lim_{x\to\infty}p(x)=0$, and therefore we obtain the constants missing in the previous equation. It is left to the reader to check that
\begin{equation}
p'(0)=-{\left(\frac{\sigma^2}{2(\gamma+\delta)}\left(1-e^{-\frac{2(\gamma+\delta)}{\sigma^2}b}\right)+\frac{\sigma^2}{2\gamma}e^{-\frac{2(\gamma+\delta)}{\sigma^2}b}\right)^{-1}}.\notag
\end{equation}
\end{proof}

\subsection{Total progeny and extinction  in the general case}\label{tpext}
In analogy with the usual continuous-state branching processes (cf. \cite{K} or \cite{DFM} in the case with immigration) we will refer to
\begin{equation}
J_t=\int_0^tV_sds,\notag
\end{equation}
as the total progeny of the refracted continuous-state branching process until time $t>0$.
 In this section we will provide distributional properties for $J_{T_a^+}$, where
\begin{equation}
T_a^+=\inf\{t>0: V_t>a\},\quad a>0.\notag
\end{equation}
A key element for the forthcoming analysis relies heavily on the theory of the  so-called scale functions for the spectrally positive L\'evy process $\overline X$
with Laplace exponent ${\overline\psi}$. Such functions are defined as follows. For
each $q\geq0$ define the function $W^{(q)}:
\mathbb{R}\to [0, \infty),$ such that $W^{(q)}(x)=0$ for all $x<0$ and on $[0,\infty)$ it is the unique continuous function with Laplace transform
\begin{eqnarray}
\int^{\infty}_0\mathrm{e}^{-\theta x}W^{(q)}(x)dx=\frac1{{\overline\psi}(\theta)-q},
\qquad \theta>{\overline\Psi}(q),
\label{scaleLT}
\end{eqnarray} 
where 
$ {\overline\Psi}(q) = \sup\{\lambda \geq 0: {\overline\psi}(\lambda) = q\}$ which is well defined and finite for all $q\geq 0$, since ${\overline\psi}$ is a strictly convex function satisfying ${\overline\psi(}0) = 0$ and ${\overline\psi}(\infty) = \infty$. For convenience, we write $W$ instead of $W^{(0)}$. Associated to the functions $W^{(q)}$ are the functions $Z^{(q)}:\R\to[1,\infty)$ defined by
\[
Z^{(q)}(x)=1+q\int_{0}^x W^{(q)} (y)\ud y,\qquad x\ge 0.
\]
Together, the functions $W^{(q)}$ and $Z^{(q)}$ are collectively known as $q$-scale functions which appear in almost all fluctuations identities for spectrally positive L\'evy processes; see \cite{K,KKR} for more details on this topic. In the same vein, for each $q\geq0$, let $\mathbb{W}^{(q)}$  denote the $q$-scale function associated to the spectrally positive L\'evy process with Laplace exponent ${\overline\psi}(\lambda) - \delta\lambda$, $\lambda\geq 0$.  We shall also write ${\overline\varphi}$ for the right inverse of this Laplace exponent;
 that is to say
\begin{equation*}
\overline\varphi(q)=\sup\{\lambda>0: {\overline\psi}(\lambda)-\delta\lambda=q\},\quad q\geq0.
\end{equation*}

\begin{theorem}\label{tp}
Let $V=\{V_t:t\geq0\}$ be the refracted continuous-state branching process defined by (\ref{rcsbp}), with underlying branching mechanism ${\overline\psi}$ satisfying that ${\overline\psi}(\infty)=\infty$. For all $a>0$ and $q\geq0$ we have:
\begin{itemize}
\item[i)] For each $x,b\in[0,a]$ with $x<b$,
\begin{equation}\label{tp1}
\mathbb{E}_x\left[e^{-q\int_{0}^{{T_0^-}}V_sds}1_{\{{T_0^-}<T_a^+\}}\right]=\frac{W^{(q)}(a-x)+\delta\int_{a-b}^{a-x}\mathbb{W}^{(q)}(a-x-y)W^{(q)\prime}(y)dy}{W^{(q)}(a)+\delta\int_{a-b}^{a}\mathbb{W}^{(q)}(a-y)W^{(q)\prime}(y)dy}.
\end{equation}
\item[ii)] For each $x,b\in[0,a]$ with $x<b$,
\begin{align}\label{tp2}
\mathbb{E}_x\left[e^{-q\int_{0}^{T_a^+}V_sds}1_{\{T_a^+<{T_0^-}\}}\right]&=Z^{(q)}(a-x)+\delta q\int_{a-b}^{a-x}\mathbb{W}^{(q)}(a-x-y)W^{(q)}(y)dy\notag\\
&-\frac{Z^{(q)}(a)+\delta q\int_{a-b}^{a}\mathbb{W}^{(q)}(a-y)W^{(q)}(y)dy}{W^{(q)}(a)+\delta\int_{a-b}^{a}\mathbb{W}^{(q)}(a-y)W^{(q)\prime}(y)dy}\notag\\
&\times\left(W^{(q)}(a-x)+\delta\int_{a-b}^{a-x}\mathbb{W}^{(q)}(a-x-y)W^{(q)\prime}(y)dy\right).
\end{align}
\end{itemize}
\end{theorem}
\begin{proof}
Let us take the refracted continuous-state branching process $V$ defined by (\ref{rcsbp}) as the Lamperti transform of the process $\overline{U}$ defined by (\ref{SDE2}) and 
let us define the following stopping times:
\begin{equation}
\overline{\kappa}_a^+=\inf\{t>0:\overline{U}_t>a\},\qquad\text{and}\qquad\overline{\kappa}_0^-=\inf\{t>0:\overline{U}_t<0\}.\notag
\end{equation}
Then by the definition of the stopping times $T_a^+$ and $T_0^-$ {and by \eqref{rcsbp}}
 it is easy to see that
\begin{equation}
\overline{\kappa}_a^+=\int_{0}^{T_a^+}V_sds,\qquad\text{and}\qquad\overline{\kappa}_0^-=\int_0^{T_0^-}V_sds.\notag
\end{equation}
We  will obtain the result by using certain identities found in \cite{KL} for refracted L\'evy processes. To do this the 
following relations are key in the rest of the proof. Consider the process $U=-\overline{U}$ defined by (\ref{SDE})
{and the process $\tilde{U}=U+a$}. Note that the following identities hold true:
\begin{align}
\overline{\kappa}_a^+&=\inf\{t>0:\overline{U}_t>a\}=\inf\{t>0:U_t<-a\}=\inf\{t>0:\tilde{U}_t<0\}\equiv\tilde{\kappa}_0^-\notag,
\end{align}
and,
\begin{align}
\overline{\kappa}_0^-&=\inf\{t>0:\overline{U}_t<0\}=\inf\{t>0:U_t>0\}=\inf\{t>0:\tilde{U}_t<a\}\equiv\tilde{\kappa}_a^+.\notag
\end{align}
Moreover,  it is not difficult to see that $\tilde{U}$ satisfies the following stochastic differential equation
\begin{equation}
\tilde{U}_t=a-x-{\overline X}_t-\delta\int_0^t1_{\{\tilde{U}_s>a-b\}}ds.\notag
\end{equation}
Finally,  putting all the pieces together and using Theorem 4 in \cite{KL} {(observe that if $V$ and then $\overline U$ is refracted at level $b$, then $\tilde U$ is refracted at level $a-b$)} we obtain
\begin{eqnarray*}
\mathbb{E}_x\left[e^{-q\int_{0}^{{T_0^-}}V_sds}1_{\{{T_0^-}<T_a^+\}}\right]
&=&\mathbb{E}_{a-x}\left[e^{-q{\tilde{\kappa}_a^+}}1_{\{{\tilde{\kappa}_a^+}<{\tilde{\kappa}_0^-}\}}\right]
\\&=&\frac{W^{(q)}(a-x)+\delta\int_{a-b}^{a-x}\mathbb{W}^{(q)}(a-x-y)W^{(q)\prime}(y)dy}{W^{(q)}(a)+\delta\int_{a-b}^{a}\mathbb{W}^{(q)}(a-y)W^{(q)\prime}(y)dy}.
\end{eqnarray*}
Similarly, we get that 
\begin{eqnarray*}
\mathbb{E}_x\left[e^{-q\int_{0}^{T_a^+}V_sds}1_{\{T_a^+<{T_0^-}\}}\right]&=&\mathbb{E}_{a-x}\left[e^{-q{{\tilde\kappa_0^-}}}1_{\{{\tilde{\kappa}_0^-}<{{\tilde\kappa_a^+}}\}}\right]
\\&=&Z^{(q)}(a-x)+\delta q\int_{a-b}^{a-x}\mathbb{W}^{(q)}(a-x-y)W^{(q)}(y)dy\notag\\
\\&&-\frac{Z^{(q)}(a)+\delta q\int_{a-b}^{a}\mathbb{W}^{(q)}(a-y)W^{(q)}(y)dy}{W^{(q)}(a)+\delta\int_{a-b}^{a}\mathbb{W}^{(q)}(a-y)W^{(q)\prime}(y)dy}\notag\\
\\&&\times\left(W^{(q)}(a-x)+\delta\int_{a-b}^{a-x}\mathbb{W}^{(q)}(a-x-y)W^{(q)\prime}(y)dy\right).\notag
\end{eqnarray*}
\end{proof}
As an immediate consequence of the last theorem we obtain the probability that the refracted continuous-state branching process passes as time goes to infinity. 
Notice that this event includes but is not always equal to the event of extinction in finite time, generally denoted by $Ext$ in the literature. 
It remains an open question to find a necessary and sufficient condition for extinction in this setting. 
\begin{corollary}\label{cor}
Under the assumptions of Theorem \ref{tp} we have for $x,b\geq 0$, and $q\geq0$:
\begin{equation}\label{ei1}
\mathbb{E}_x\left[e^{-q\int_{0}^{{T_0^-}}V_sds}1_{\{T_0^-<\infty\}}\right]=\frac{e^{-{\overline\Psi}(q)x}+\delta {\overline\Psi}(q)\int_{x}^{b}\mathbb{W}^{(q)}(u-x)e^{-{\overline\Psi}(q)u}du}{1+\delta{\overline\Psi}(q)\int_{0}^{b}\mathbb{W}^{(q)}(u)e^{-{\overline\Psi}(q)u}du}.
\end{equation}
In particular,
\begin{equation}
\mathbb{P}_x\left(\lim_{t\to\infty}V_t=0\right)=\frac{e^{-{\overline\Psi}(0)x}+\delta {\overline\Psi}(0)\int_{x}^{b}\mathbb{W}(u-x)e^{-{\overline\Psi}(0)u}du}{1+\delta{\overline\Psi}(0)\int_{0}^{b}\mathbb{W}(u)e^{-{\overline\Psi}(0)u}du}.\notag
\end{equation}
\end{corollary}
\begin{proof}
The proof of the first part follows by taking limits when $a\uparrow\infty$ in identity (\ref{tp1}), and by noting that (cf. Exercise 8.5 in \cite{K})
\begin{equation}
\lim_{a\uparrow\infty}\frac{W^{(q)}(a-x)}{W^{(q)}(a)}=e^{-{\overline\Psi}(q)x}\qquad\text{and}\qquad\lim_{a\uparrow\infty}\frac{W^{(q)\prime}(a)}{W^{(q)}(a)}={\overline\Psi}(q).\text{ for $x,q\geq0$}\notag
\end{equation}
The second part follows by letting  $q\downarrow0$ in (\ref{ei1}).
\end{proof}
We close this section with a result that gives the probability that the refracted continuous-state branching process stays below a given level $a>0$.
\begin{corollary}
Under the assumptions of Theorem \ref{tp} we have for $a\geq x,b\geq 0$:
\begin{align}
\mathbb{P}_x(\sup_{s\leq\infty}V_s{\leq} a)&=\frac{W(a-x)+\delta\int_{a-b}^{a-x}\mathbb{W}(a-x-y)W^{\prime}(y)dy}{W(a)+\delta\int_{a-b}^{a}\mathbb{W}(a-y)W^{\prime}(y)dy}.\notag
\end{align}
In particular we have that

\begin{align}
\mathbb{P}_x(\sup_{s\leq\infty}V_s<\infty)=\mathbb{P}_x\left(\lim_{t\to\infty}V_t=0\right)\notag
\end{align}
as provided in Corollary \ref{cor}.
\end{corollary}
\begin{proof}
The proof of the first part of the corollary follows by taking $q\downarrow0$ in identity (\ref{tp1}), while the second part follows by taking $a\uparrow\infty$ in the first claim of the corollary.
\end{proof}

\subsection{Explosion}
In this section we will obtain certain conditions that guarantee that  the event of explosion for the refracted continuous-state branching process $V$ has positive probability. Following \cite{cpu} for any c\`adl\`ag function $f$ without negative jumps we will refer to the following initial problem for the function $c(t)=\int_0^tf(s)ds$.
\begin{equation}
IVP(f)=
\begin{cases}
c'(t)_+=f\circ c(t), 			 \ \ \ \ \ \ \ \\
c(0)=0, \ \ \ \ \ \ \  \\
\end{cases}
\end{equation}
In the case without  refraction the criteria for explosion  is due to Ogura \cite{og} and Grey \cite{gr}, who assert that the probability that a continuous-state branching process with Laplace exponent ${\psi}$ started from $x>0$ is absorbed at infinity in finite time is positive if and only if 
\begin{equation}
\int_{0+}\frac{1}{\psi(\lambda)}d\lambda>-\infty.\notag
\end{equation}
We call such $\psi$ an explosive branching mechanism. Let us  denote by $\overline X^{\delta}$ the L\'evy process 
\begin{equation}
\overline X^{\delta}_t={\overline X}_t+\delta t,\qquad\text{ for $t\geq0$,}\notag
\end{equation}
and by $\overline\psi^{\delta}$ its Laplace exponent. We have the following result which is analogous to Corollary 5 in \cite{cpu} and, as such the proof goes along the same lines.
\begin{proposition} 
Let $x>0$.
\begin{itemize}
\item[(1)] The probability that a RCSBP($\overline\psi$) $V$ that starts at $x$ jumps to $\infty$ is positive if and only if ${\overline\psi}(0)\not=0$.
\item[(2)] The probability that $V$ reaches $\infty$ continuously is positive if ${\overline\psi}(0)=0$ and ${\overline\psi}$ is an explosive mechanism.
\item[(3)] We have that ${\overline\psi}(0)=0$ and ${\overline\psi^{\delta}}$ is an explosive mechanism if the probability that $V$ reaches $\infty$ continuously is positive.
\end{itemize}
\end{proposition}
\begin{proof} 
For (1), it is easy to see that $V$ jumps to infinity if and only if ${\overline U}$ jumps to infinity and never reaches $-x$. This event has positive probability if and only if ${\overline\psi}(0)\not=0$.

For (2), assume that ${\overline\psi}(0)=0$ and ${\overline\psi}$ is an explosive mechanism. Let us consider $\overline{C}$ the unique solution to IVP($x-\varepsilon+{\overline X}$) for $\varepsilon>0$ such that $x-\varepsilon>0$. In this case the right-hand derivative of $\overline{C}$ which we will denote by $Z$ is a CB(${\overline\psi}$) started at $x-\varepsilon>0$. Consider that $\overline{C}$ explodes in finite time, let say $\tau$. On the other hand let us denote by $C$ the unique solution to IVP($x+\overline U$), then using Lemma 2 in \cite{cpu} we have that $\overline{C}\leq C$. This in turn implies that $C$ explodes in $(0,\tau)$. So we have that
\begin{equation}
\mathbb{P}(\text{$Z$ reaches $\infty$ continuously })\leq \mathbb{P}(\text{$V$ reaches $\infty$ continuously}).\notag
\end{equation}
Hence,  by the Ogura-Grey explosion criterion we have that $\overline{C}$ explodes in finite time with positive probability, and the result follows.

To prove (3) we take $\varepsilon>0$, and consider $\tilde{C}$ the unique solution to IVP($x+\varepsilon+\overline X^{\delta}$) and denote its right-hand derivative by $\tilde{Z}$ which is a CB($\overline \psi^{\delta}$) started at $x+\varepsilon>0$. Then by Lemma 2 in \cite{cpu} we have that $C\leq \tilde{C}$ and, as in the previous point we have that
\begin{equation}
\mathbb{P}(\text{$V$ reaches $\infty$ continuously})\leq \mathbb{P}(\text{$\tilde{Z}$ reaches $\infty$ continuously}).\notag
\end{equation}
Therefore we have that if the probability that $V$ explodes in finite time is positive then it is positive for the process $\tilde{Z}$. Finally, by the converse of the Ogura-Grey criterion we have that $\overline\psi(0)=0$ and $\overline\psi^{\delta}$ is explosive.
\end{proof}


\section{Law of iterated logarithm and maximal jump for a RCSBP}\label{illmj}

In this section we will extend classical results for continuous-state branching processes to the case of refracted continuous-state branching processes.  
Those results are obtained thanks to preliminary works in \cite{cpu} and \cite{be2}.

\subsection{Law of iterated logarithm}
The first result involves {\it translating} a law of iterated logarithm for L\'evy processes to the corresponding refracted continuous-state branching process. This is done  throughout  the representation of the latter  as time changed refracted L\'evy process  (\ref{rcsbp}).
\begin{proposition}
Let $V=\{V_t:t\geq 0\}$ be the RCSBP defined in (\ref{rcsbp}) that starts at $x>0$, let $\overline\Psi$ be the right-continuous inverse of ${\overline\psi}$ given in (\ref{lk}), and define
\begin{equation}
\alpha(t)=\frac{\log|\log t|}{{\overline\Psi}(t^{-1}\log|\log t|)}.\notag
\end{equation}
Then there exists a constant $\xi$ (in general nonzero) such that
\begin{equation}
\liminf_{t\to 0}\frac{V_t-x}{\alpha(tx)}=\xi+\delta\overline{d}1_{\{x<b\}},\notag
\end{equation}
where $\overline{d}$ is the drift coefficient of the subordinator with Laplace exponent given by ${\overline\Psi}$.
\end{proposition}
\begin{proof}
Let $\overline X$ be a spectrally positive L\'evy process with Laplace exponent $\overline{\psi}$.
Following Caballero et al. \cite{cpu} the behaviour of the function $\alpha$ near zero is determined by the value of the drift coefficient $\overline{d}$. 
More precisely, as $t\to0+$
\begin{equation}\label{lil3}
\begin{cases} 
{\frac{1}{\alpha(t)}\sim \frac{\overline{d}}{t}}, 			 \ \ \ \ \ \ \     \text{if $\overline{d}>0$},\\
{\frac{1}{\alpha(t)}=o(\frac{1}{t})}, \ \ \ \ \ \ \  \text{if $\overline{d}=0$.} \\
\end{cases}
\end{equation}
On the other hand it is also a known result (\cite{cpu}) that if $a$ is a function such that $a_t\to 1$ as $t\to 0$, then
\begin{equation}\label{lil1}
\lim_{t\to 0}=\frac{\alpha(a_tt)}{\alpha(t)}=1.
\end{equation}
Fristedt and Pruitt \cite{fp} proved the existence of a constant $\xi\not=0$ such that
\begin{equation*}
\liminf_{t\to 0} \frac{\overline X_t}{\alpha(t)}=\xi.\notag
\end{equation*}
Now let us recall that the process $V$ is the unique solution to
\begin{align}\label{lil2}
V_t={\overline U}_{\int_0^tV_sds}={\overline X}_{\int_0^tV_sds}+\delta\int_0^tV_s1_{\{V_s<b\}}ds.
\end{align}
Using that $V_0=x$, and that $V$ is right-continuous, then
\begin{equation*}
\lim_{t\to 0+}\frac{1}{t}\int_0^tV_sds=x.\notag
\end{equation*}
This implies using (\ref{lil1}) that
\begin{equation*}
\lim_{t\to 0+}\frac{\alpha(\int_0^tV_sds)}{\alpha(xt)}=1,\notag
\end{equation*}
and therefore
\begin{equation}
\liminf_{t\to 0+}\frac{{\overline X}_{\int_0^tV_sds}}{\alpha(xt)}=\xi.\notag
\end{equation}
Now let us work with the integral term in (\ref{lil2}), to this end we note that
\begin{equation}
\lim_{t\to 0+}\frac{1}{t}\int_0^tV_s1_{\{V_s<b\}}ds=x1_{\{x<b\}}.\notag
\end{equation}
Therefore using (\ref{lil3}) we obtain that
\begin{equation}
\lim_{t\to 0+}\frac{\int_0^tV_s1_{\{V_s<b\}}ds}{\alpha(xt)}=\overline{d}1_{\{x<b\}}.\notag
\end{equation}
So putting the pieces together we obtain
\begin{equation}
\liminf_{t\to 0+}\frac{V_t-x}{\alpha(xt)}=\xi+\delta\overline{d}1_{\{x<b\}}.\notag
\end{equation}
\end{proof}

\subsection{The maximal jump}
We close this section with a result about the maximal jump of the refracted continuous-state branching process  $V$ given by (\ref{rcsbp}), which is done following some ideas from \cite{be2}. Let us introduce the following notation
\[
\Delta^*=\sup\{\Delta V_t:0\leq t\leq T_0^-\},
\]
and
\[
\delta_y=\inf\{t\geq0:\Delta\overline{U}_t>y\}=\inf\{t\geq0:\Delta\overline X_t>y\}.
\]
It is known that $\delta_y$ has an exponential distribution of parameter $\overline\Pi^*(y):=\int_y^\infty\overline\Pi(dz)$. Now we define the process
\[
\overline{X}^y_t=\overline{X}_t-\sum_{0\leq s\leq t}\Delta\overline{X}_s1_{\{\Delta\overline{X}_s>y\}},\qquad\text{$t\geq0$}
\]
and we consider its associated {process $\overline U^y$}, given as the solution to the following SDE
\[
\overline{U}^y_t=\overline{X}^y_t+\delta\int_0^t\mathbf{1}_{\{\overline{U}^y_s<b\}}ds, \qquad t\geq 0.
\]
As usual we denote by $\kappa^{y}_0=\inf\{t\geq0:\overline{U}^y_t<0\}$ and note that we have the following equality of events
\[
\{\Delta^*\leq y\}=\{\kappa^{y}_0\leq \delta_y\}.
\]
Then it is easy to see by the L\'evy-It\^o decomposition that $\delta_y$ is independent of ${\overline U}^y$ and $\kappa^{y}_0$, therefore we have the following result:
\[
\mathbb{P}_x(\Delta^*\leq y)=\mathbb{P}_x(\kappa^{y}_0<\delta_y)=\mathbb{E}_x(\exp\{-{\overline\Pi^*(y)}\kappa^{y}_0\}).
\]
In order to compute the later expectation we note the following
\[
\kappa^{y}_0=\inf\{t\geq0:\overline{U}^y_t<0\}=\inf\{t\geq0:U^y_t>0\},
\]
where ${U^y=-\overline{U}^y}$.

Therefore by Theorem 5 (i) in \cite{KL} we have that
\begin{align}
\mathbb{E}_x(\exp\{-\overline\Pi^*(y) \kappa^{y}_0\})=\frac{e^{-\overline\phi^y(\overline\Pi^*(y))x}+\delta  \overline\phi^y(\overline\Pi^*(y))\int_{-b}^{-x} e^{\overline\phi^y(\overline\Pi^*(y))z}\mathbb{W}^{(\overline{\Pi}^*(y))}(-x-z)dz}{1+\delta  \overline\phi^y(\overline\Pi^*(y))\int_{-b}^{0}e^{\overline\phi^y(\overline\Pi^*(y))z}\mathbb{W}^{(\overline\Pi^*(y))}(-z)dz},\notag
\end{align}
where ${\overline\phi^y}$ is the inverse of the Laplace exponent of the spectrally negative L\'evy process $-\overline{X}^y$. 
Now if we define by ${\overline\psi^y}$ the Laplace exponent of the process ${\overline{X}^y}$ then it is easy to see that
\[
{\overline\psi^y(\theta)=\overline\psi(\theta)+\int_y^{\infty}(1-e^{\theta z})\overline\Pi(dz).}
\]
Therefore we have the following result:
\begin{proposition}
For any $y>0$ we have
\begin{align}
\mathbb{P}_x(\Delta^*\leq y)&=\frac{e^{-\theta^yx}+\delta \theta^y\int_{-b}^{-x}e^{\theta^yz}\mathbb{W}^{(\overline{\Pi}^*(y))}(-x-z)dz}{1+\delta \theta^y\int_{-b}^{0}e^{\theta^yz}\mathbb{W}^{(\overline{\Pi}^*(y))}(-z)dz}\notag\\
& =\frac{e^{-\theta^yx}+\delta \theta^y\int_{x}^{b}e^{-\theta^yz}\mathbb{W}^{(\overline{\Pi}^*(y))}(z-x)dz}{1+\delta \theta^y\int_{0}^{b}e^{-\theta^yz}\mathbb{W}^{(\overline{\Pi}^*(y))}(z)dz},
\end{align}
where $\theta^y=\inf\{\theta\geq0: {\overline\psi(\theta)=\int_y^{\infty}e^{\theta z}\overline\Pi(dz)}\}$.
\end{proposition}


\begin{thebibliography}{99}




\bibitem{B} {\sc Bertoin, J.} {\it L\'evy processes.} {\rm Cambridge University Press}, (1996).


\bibitem{be2} {\sc Bertoin, J.} {On the maximal offspring in a critical branching process with infinite variance.} {\it J. Appl. Prob.} {\bf  48}, 576--582, (2011).

\bibitem{cpu} {\sc Caballero, M.E., P\'erez-Garmendia, J.L. and Uribe-Bravo, G.} {A Lamperti-type representation of continuous-state branching processes with immigration.} {\it Ann. Probab.} {\bf  41}, 1585--1627, (2013).


\bibitem{DFM} {\sc Duhalde, X., Foucart, C. and Ma, C.} \rm On the hitting times of continuous-state branching processes with immigration.
{\it  Stochastic Process. Appl.} {\bf 124}, 418--4201, (2014). 


\bibitem{fp} \sc Fristedt, B.E. and Pruitt, W.E. \rm Lower functions for increasing random walks
and subordinators. {\it  Z. Wahrsch. Verw. Gebiete.} {\bf 18}, 167--182, (1971).  

\bibitem{gr} \sc Grey, D.R. \rm Asymptotic behaviour of continuous time, continuous state-space branching processes. {\it J. Appl. Probab.} {\bf 11}, 669--677, (1974).



\bibitem{JK}{\sc  Jagers, P., Klebaner, F.C.} {\rm Dependence and Interaction in Branching
Processes}. {\it In  Prokhorov and contemporary probability theory}, Springer Proc. Math. Stat. 33, Springer, (2013).


\bibitem{KKR}\sc  Kuznetsov, A., Kyprianou, A.E., Rivero, V. \rm The
theory of scale functions for spectrally negative L\'evy processes. {\it In L\'evy Matters II}, Lecture Notes in Mathematics, Springer, (2012).







\bibitem{K} \sc Kyprianou, A.E. {\it Introductory lectures on
fluctuations of L\'evy processes with applications.} \rm Springer, (2006).


\bibitem{KL}\sc Kyprianou, A.E., Loeffen, R. \rm Refracted L\'evy
processes. {\it Ann. Inst. H. Poincar\'e} {\bf46}(1), 24--44, (2010).








\bibitem{la1} {\sc Lamperti, J}. {\rm Continuous state branching processes.} {\it Bull. Amer. Math. Soc.} {\bf73}, 382-€"-386, (1967).


\bibitem{LP}{\sc Li, Z., Pu, F.} {\rm Strong solutions of jump-type stochastic equations.} {\it Electron. Commun. Probab.} {\bf 17}(33), 4--13, (2012).


\bibitem{Pr} {\sc Protter, P.E.} {\it Stochastic Integration and Differential Equations}, 2nd Ed., Springer, (2005).


\bibitem{og}\sc Ogura, Y. \rm Spectral representation for branching process on the real half line. {\it Res. Inst. Math. Sci.} {\bf 5}, 423--441, (1969/1970).





\end{thebibliography}
\end{document}